\documentclass[leqno]{amsart}

\usepackage{amsmath}
\usepackage{amsfonts}
\usepackage{amssymb}
\usepackage{amsthm}
\usepackage{mathrsfs}
\usepackage{hyperref}
\usepackage{enumerate}
\usepackage{cleveref}
\usepackage{calc}

\usepackage{graphicx} 

\newcommand\labN[2][12pt]{\makebox(0,0)[cb]{\raisebox{#1 -3pt + .5\totalheight  }{\mbox{#2}}}}

\newcommand\labS[2][12pt]{\makebox(0,0)[ct]{\raisebox{-#1-\height}{\mbox{#2}}}}

\newcommand\labnhS[2][3pt]{\labS[#1]{#2}}  

\newtheorem{theorem}{Theorem}
\newtheorem*{theorem*}{Theorem}

\newtheorem*{lemma*}{Lemma}

\theoremstyle{definition}

\newtheorem*{remark*}{Remark}

\theoremstyle{plain}

\newcommand{\Z}{\mathbb{Z}}

\newcommand{\Q}{\mathbb{Q}}
\newcommand{\R}{\mathbb{R}}
\newcommand{\C}{\mathbb{C}}

\newcommand{{\D}}{\delta}

\newcommand{\eps}{\varepsilon}

\newcommand{\SL}{\operatorname{SL}}

\newcommand{\PSL}{\operatorname{PSL}}

\newcommand{\sign}{\operatorname{sgn}}
\renewcommand{\Re}{\operatorname{Re}}
\renewcommand{\Im}{\operatorname{Im}}

\newcommand{\SLZ}{\SL_2(\Z)}
\newcommand{\abcd}{\left(\begin{smallmatrix} a & b \\ c & d \end{smallmatrix}\right)}
\newcommand{\bigabcd}{\begin{pmatrix} a & b \\ c & d \end{pmatrix}}

\newcommand{\HH}{\mathfrak{H}}

\newcommand{\abs}[1]{\left\vert#1\right\vert}

\numberwithin{equation}{section}
\numberwithin{table}{section}

\author{Michael H. Mertens}
\address{Universit\"at zu K\"oln, Department Mathematik/Informatik, Weyertal 86--90, 50931 K\"oln, Germany} 
\email{mmertens@math.uni-koeln.de}
\author{Mark A. Norfleet}
\address{University of Wisconsin--Green Bay} 
\email{norfleem@uwgb.edu}

\title[A multiplier system for $\Gamma_0^+(p)$ and a geometric formulation]{Weight $1/2$ multiplier systems for the group $\Gamma_0^+(p)$ and a geometric formulation}

\begin{document}
\maketitle

\begin{abstract}
We construct a weight $1/2$ multiplier system for the group $\Gamma_0^+(p)$, the normalizer of the congruence subgroup $\Gamma_0(p)$ where $p$ is an odd prime, and we define an analogue of the eta function and Rademacher symbol and relate it to the geometry of edge paths in a triangulation of the upper half plane.
\end{abstract}

\section{Introduction}

A very important modular form is given by the famous \emph{Dedekind eta function}, which is usually defined as the infinite product 
\[
\eta(z):=q^{1/24}\prod_{n=1}^\infty (1-q^n), \qquad q= e^{2 \pi i z},
\]
for $z$ in the upper half-plane $\HH=\{z\in\C\: :\: \Im(z)>0\}$. This function enters into mathematics in many different places, one of which is the famous \emph{first Kronecker limit formula} (see e.g. \cite[Chapter 1]{Siegel61}): The \emph{Eisenstein series}
\[
E(z,s):=\sum_{\substack{m,n\in\Z \\ (m,n)\neq 0}} \frac{\Im(z)}{|mz+n|^{2s}},\quad \Re(s)>1,
\]
has a meromorphic continuation to the entire $s$-plane which is holomorphic except for a simple pole in $s=1$ and one has
\[
\lim_{s\to 1}\left(E(z,s)-\frac{\pi}{s-1}\right)=2\pi\left(\gamma_E-\log(2)-\log\left(\sqrt{\Im(z)}|\eta(z)|^2\right)\right),
\]
where $\gamma_E=0.57721566...$ denotes the Euler-Mascheroni constant. 

Since the Eisenstein series is invariant under \emph{M\"obius transformations} $z\mapsto \abcd.z:=\frac{az+b}{cz+d}$ with $\abcd\in \SLZ$, Kronecker's limit formula allows to deduce the well-known fact that $\eta$ satisfies the transformation law
\[
\eta\left(\frac{a z+b}{cz+d}\right)=\eps(a,b,c,d)(cz+d)^{1/2}\eta(z),\quad \abcd\in\SLZ,
\]
for some $\eps(a,b,c,d)\in\C$. In particular, we find
\[
\eta(z+1) = e^\frac{\pi i}{12}\eta(z),\qquad\qquad
\eta\left(-\frac 1z\right) = \sqrt{\frac{z}{i}}\eta(z),
\]
where we choose the branch of the square-root that is positive for positive real arguments.  Since these transformations generate the full modular group, this implies that $\eta$ is a modular form of weight $1/2$ for the group $\SLZ$ with respect to a certain multiplier system (see below). Dedekind was the first to obtain an explicit description of this multiplier system in terms of so-called \emph{Dedekind sums}: For coprime integers $h,k$, we let
\[
s(h,k)=\sum_{\mu=1}^k \left(\hspace{-4pt}\left(\frac{h\mu}{k}\right)\hspace{-4pt}\right)\left(\hspace{-4pt}\left(\frac\mu k\right)\hspace{-4pt}\right), %
\quad \text{ where } \quad %
(\hspace{-1pt}(x)\hspace{-1pt}):=%
\begin{cases}%
x-\lfloor x\rfloor -\frac 12 & \text{if }x\notin\Z ,\\%
0 & \text{if }x\in\Z.%
\end{cases}%
\]
With those we define the \emph{Rademacher symbol} for a matrix $\gamma=\abcd\in\SLZ$ to be
\[
\Phi(\gamma):=%
\begin{cases}%
\frac{b}{d} & \text{if }c=0, \\%
\frac{a+d}{c}-12\sign(c)s(d,|c|)  & \text{if }c\neq 0.%
\end{cases}%
\]
It can be shown that $\Phi(\gamma)$ is always an integer, see for instance p. 50 in \cite{RadGros}. With this we have for $\abcd\in\SLZ$ that
\begin{equation}\label{etatrans}
\log \eta\left(\frac{az+b}{cz+d}\right)%
=%
\log\eta(z)%
+ \frac{1}{2}\sign(c)^2 \log\left(\frac{cz+d}{i\sign(c)}\right)%
+ \frac{\pi i}{12}\Phi\left(\bigabcd\right),
\end{equation}
where the second summand is understood to be $0$ if $c=0$.  Note that throughout we understand $\sign(0):=0$ The Rademacher symbol almost behaves like a logarithm on $\SLZ$,
\[
\Phi(\gamma_1\gamma_2) = \Phi(\gamma_1) + \Phi(\gamma_2) - 3\sign(c_1c_2c_3),
\]
where $\gamma_3=\gamma_1\gamma_2$ and %
$\gamma_j = \left(\begin{smallmatrix} a_j & b_j \\ c_j & d_j \end{smallmatrix}\right)$, %
which tells us that the Rademacher symbol is essentially the logarithm of a \emph{multiplier system} of weight $1/2$ for the full modular group. For an exact definition of this term we refer the reader to standard textbooks on the theory of modular forms, e.g. Section 2.6 in \cite{Iwaniec}. 

These Dedekind sums (also Rademacher symbols) are therefore very natural objects connected to the group $\SLZ$. They have appeared in many different contexts in number theory, but also in geometry and topology. For example, Dedekind sums are present in signature related invariants of lens spaces, such as their $\alpha$-invariants in \cite{HirzebruchZagier74} and $\mu$-invariants in \cite{NeumannRaymond78}; they also appear in the study of signatures of torus bundles over surfaces (see \cite{Meyer73}) and generalized Casson invariant (see \cite{Walker92}). Furthermore, Kirby and Melvin in \cite{KirbyMelvin94} give a geometric definition of Rademacher symbols. Their geometric definition of the Rademacher symbol arises from the action of the modular group on the upper half plane $\HH$ by fractional linear transformations.  Their geometric definition is based on a certain triangulation $K$ of $\HH$ by ideal triangles obtained by successive reflections of the ideal triangle with vertices at $0$, $1$, and $\infty$ on the boundary of $\HH$. There are many special properties of this triangulation $K$; in \cite{HatcherBook}, this triangulation is used to elucidate a geometric view of Pythagorean triples, the Euclidean algorithm, Pell’s equation, continued fractions, and Farey sequences. 

In section \ref{TriInHyperbolicPlane}, we recall how elements of the modular group can be related to edges in the triangulation $K$ by the well known fact that the edges of the triangulation have endpoints $a/c$ and $b/d$ if and only if $ad-bc = \pm 1$. Next, starting with the edge from $\infty$ to $0$, we associate a (directed) based edge path $\alpha$ in the triangulation $K$ that ends at the edge related to $\abcd$ in the modular group.  With these based edge paths in the triangulation $K$ of $\HH$, Kirby and Melvin give a geometric formulation of Rademacher symbol \cite{KirbyMelvin94}. 

\begin{theorem*}[Kirby \& Melvin]
For $A= \left( \begin{smallmatrix} a &  b \\ c & d \end{smallmatrix} \right) = S(T^{a_1} S) \cdots (T^{a_k} S)$ in the modular group with generators $S = \left( \begin{smallmatrix} 0 &  -1 \\ 1 & 0 \end{smallmatrix} \right)$ and $T =  \left( \begin{smallmatrix} 1 &  1 \\ 0 & 1 \end{smallmatrix} \right)$, and set $\alpha = \left( a_1 , \ldots , a_k \right)$,  we have the identity
\[
\Phi(A) = \tau_\alpha - 3 \sigma_\alpha,
\]
where $\tau_\alpha$ and $\sigma_\alpha$ denotes the trace and signature of the matrix
\[
\left(\begin{smallmatrix}
a_1   &   1   &        & & &        &        \\
1     &   a_2 &   1    & & &        &        \\
      &   1   & \ddots & & &        &       \\
      &       &        & & &        &       \\
      &       &        & & &\ddots  &    1   \\
      &       &        & & &  1     &   a_k  
\end{smallmatrix}\right)
\]
whose $(i,j)^{\text{th}}$ entry is $a_i$ if $i = j$, $1$ if $\abs{i-j}=1$, and $0$ otherwise. 
\end{theorem*}

Both $\tau_\alpha$ and $\sigma_\alpha$ have geometric interpretations in terms of the based edge path $\alpha$ in the triangulation $K$. Informally speaking, $\tau_\alpha$ is the net number of turns made in $\alpha$, keeping in mind the direction, and $\sigma_\alpha$ is the difference between the number of edges of $\alpha$ pointing left and those pointing right (see \cite[Remark 1.13(c)]{KirbyMelvin94}).  

In fact, Kirby and Melvin use this geometric formulation to illuminate the appearance of Dedekind sums (and Rademacher symbols) in topology; moreover, the connection of Rademacher symbols with linking numbers of trefoil knots \cite{Ghys07} has been generalized from the modular point of view by Duke, Imamoglu, and T\'{o}th \cite{DIT17}. 

On the number theoretic side, the Kronecker limit formula has been generalized to arbitrary Fuchsian groups of the first kind (i.e.  discrete subgroups of $\SL_2(\R)$ of finite covolume); moreover, the corresponding analogue of the Dedekind eta function gives rise to analogous definitions of Dedekind sums and Rademacher symbols for such groups, which are made explicit for the special case of the principal congruence subgroups in \cite{Goldstein73}. Furthermore, arithmetic and analytic properties of these generalizations have been studied for instance in \cite{Burrin17,Burrin18,Vardi87,Vardi93} (see also \cite{Burrin19} for a survey of these results).

The current work generalizes the aforementioned work of Kirby and Melvin to a wider class of arithmetic groups, namely the groups $\Gamma_0^+(p)$ for odd primes $p$, the normalizers in $\SL_2(\R)$ of the congruence subgroups 
\[
\Gamma_0(p) = \left\{\bigabcd\in\SLZ \: :\: c\equiv 0\pmod p\right\}.
\]

To this end, we introduce a Rademacher symbol for $\Gamma_0^+(p)$ and show the following result.
\begin{theorem} \label{thmPhiP}
For $\gamma=\abcd\in\Gamma_0^+(p)$, let
\[
\Phi_p(\gamma):=
\begin{cases} 
\frac{1}{2} \left[ \Phi(\gamma) + \Phi \left( %
\left( \begin{smallmatrix} %
a   & pb \\ %
c/p & d
\end{smallmatrix} \right) %
\right) \right] & \text{for }\gamma\in\Gamma_0(p),\\
\Phi_p \left( %
\frac{1}{\sqrt{p}} \left( \begin{smallmatrix} %
  c  & d \\ %
 -pa & -pb %
\end{smallmatrix} \right) %
\right) - 3\sign(-ac) & \text{for }\gamma\notin\Gamma_0(p),
\end{cases}
\]
and %
\[
\eta_p(z):=\left( \eta(z)^k \eta( p z)^k \right)^{1/2k},
\]
for a suitable positive integer $k$ (see Section~\ref{Symbolp}). Then we have for all $\gamma=\abcd\in\Gamma_0^+(p)$ and all $z\in\HH$ that
\[
\log \eta_p\left(\frac{az+b}{cz+d}\right) = %
\log\eta_p(z) + \frac 12\sign(c)^2\log\left(\frac{cz+d}{i\sign(c)}\right) + \frac{\pi i}{12}\Phi_p(\gamma),
\]
where the second term is to be interpreted as $0$ if $c=0$ as in \eqref{etatrans}.
\end{theorem}

With this and the geometry of based edge paths, we are able to establish an interpretation of the Rademacher symbols in terms of averaging net turns and directions arising from the geometry from two based edge paths in $K$; we show the following result.
\begin{theorem}\label{thmmain}
For $\gamma=\abcd \in \Gamma_0^+(p)$, we have that
\[
\Phi_p(\gamma):=
\begin{cases} 
\frac{\tau_\alpha + \tau_\beta}{2} - 3 \left( \frac{\sigma_\alpha + \sigma_\beta}{2} \right) & \text{for } \gamma \in\Gamma_0(p),\\
\Phi_p \left( %
\frac{1}{\sqrt{p}} \left( \begin{smallmatrix} %
c   & d \\ %
-pa & -pb %
\end{smallmatrix} \right)%
\right) -  3\sign(-ac) %
& \text{for } \gamma \notin \Gamma_0(p),
\end{cases}
\]
where $\alpha = ( a_1, \ldots , a_k )$ and $\beta = (b_1 , \ldots b_m)$ are two based edge paths in $K$ made from $\gamma \in \Gamma_0(p)$ and %
$\left(\begin{smallmatrix} %
a & pb \\ %
c/p & d %
\end{smallmatrix}\right)\in\SLZ$, $\tau_\alpha$ and $\sigma_\alpha$ denotes the trace and signature of the matrix $M_\alpha$, and $\tau_\beta$ and $\sigma_\beta$ denotes the trace and signature of the matrix $M_\beta$
\[
M_\alpha = \left(\begin{smallmatrix}
a_1   &   1   &        & & &        &        \\
1     &   a_2 &   1    & & &        &        \\
      &   1   & \ddots & & &        &       \\
      &       &        & & &        &       \\
      &       &        & & &\ddots  &    1   \\
      &       &        & & &  1     &   a_k  
\end{smallmatrix}\right)
\quad \text{and} \quad %
M_\beta =\left( \begin{smallmatrix}
b_1   &   1   &        & & &        &        \\
1     &   b_2 &   1    & & &        &        \\
      &   1   & \ddots & & &        &       \\
      &       &        & & &        &       \\
      &       &        & & &\ddots  &    1   \\
      &       &        & & &  1     &   b_m  
\end{smallmatrix}\right)
\]
For $M_\alpha$ (resp. $M_\beta$) the $(i,j)^{\text{th}}$ entry is $a_i$ (resp. $b_i$) if $i = j$, and for either matrix the $(i,j)^{\text{th}}$ entry $1$ if $\abs{i-j}=1$, and $0$ otherwise.
\end{theorem}

The rest of the paper is organized as follows: In Section~\ref{Symbolp}, we motivate and prove Theorem~\ref{thmPhiP}; in Section \ref{TriInHyperbolicPlane}, we recall needed geometric properties of edge paths in the triangulation $K$ from \cite{KirbyMelvin94}; and in Section~\ref{ProofGeomThm}, we show how two based edge path arise in $K$ from elements of $\Gamma_0(p)$ and prove Theorem~\ref{thmmain}.

\section{Modular construction of multipliers for $\Gamma_0^+(p)$} \label{Symbolp}
The weight $1/2$ multiplier system defined by the transformation law of the eta function is a very natural one for example in the following sense: The $24^{\rm th}$ power of $\eta$ yields the unique normalized cusp form of lowest possible weight (12 in this case) for $\SLZ$. This function is usually denoted by $\Delta$ and has the important property that it never vanishes on the upper half-plane; so taking the $24^{\rm th}$ root of $\Delta$ is well-defined, once a branch of the logarithm is chosen.

The group $\Gamma_0^+(p)$ for an odd prime $p$ is well known (see e.g. \cite{Conway}) to be generated by $\Gamma_0(p)$ and the \emph{Fricke involution} %
$ W_p = %
\frac{1}{\sqrt{p}} %
\left( \begin{smallmatrix} %
 0 & -1 \\ %
 p & 0 \end{smallmatrix}\right)$.

As an analogue of the Delta function, we take any normalized (non-trivial) cusp form for $\Gamma_0^+(p)$ of lowest possible integral weight $k$ that does not vanish in $\HH$. It follows immediately from work of Kohnen \cite{Kohnen} that such a function is necessarily given by an \emph{eta quotient}. An eta quotient of level $N$ is an expression of the form %
$ f(z) = \prod_{d|N} \eta(dz)^{r_d}$, %
where $r_d\in\Z$. In general, such a function will be a weakly holomorphic modular form\footnote{This means it is holomorphic on $\HH$, but it might have poles at the cusps.} of weight $\tfrac 12 \sum_{d|N}r_d$ for some congruence subgroup of $\SLZ$. In \cite{Newman1, Newman2}, Newman gives explicit conditions on the exponents $r_d$ that make an eta quotient a holomorphic modular form for the group $\Gamma_0(N)$. Additionally, our desired cusp form should be invariant under the Fricke involution $W_p$, which finally implies that the desired analogue of the Delta function is given by %
$ \Delta_p(z):=\eta^k(z)\eta^k(pz) $, %
where $k$ is the smallest positive even integer such that $\frac{p-1}{24}k$ is an integer. Taking the $2k^{\rm th}$ root of this yields the definition of our analogue of the eta function
\[
\eta_p(z)=\exp\left(\frac 1{2k}\log(\Delta_p(z))\right),
\]
where we choose the principal branch of the logarithm, i.e.  $\log z=\log|z|+i\arg z$
for $z\in\C\setminus\{0\}$, where we pick $\arg z\in (-\pi,\pi]$.

With this we can now prove Theorem~\ref{thmPhiP}.
\begin{proof}[Proof of Theorem~\ref{thmPhiP}]
We first assume that $\gamma=\abcd\in\Gamma_0(p)$. Then we have 
\[\eta(p\gamma z)=\eta\left(\begin{pmatrix}
a & pb \\ c/p & d
\end{pmatrix}(pz) \right),\]
hence the claimed transformation law follows by directly plugging into \eqref{etatrans}, since the matrix $\left(\begin{smallmatrix}
a & pb \\ c/p & d
\end{smallmatrix}\right)$ is in $\SLZ$.

Now let $\gamma=\abcd\in\Gamma_0^+(p)\setminus\Gamma_0(p)$; we have $\gamma = W_p \tilde \gamma$, where 
\[
\tilde\gamma = W_p^{-1} \gamma =  \frac{1}{\sqrt{p}}\begin{pmatrix}
c & d \\ -pa & -pb
\end{pmatrix} \in \Gamma_0(p).
\]
From the transformation formula $\eta\left(-\frac 1z\right) = \sqrt{\frac{z}{i}}\eta(z)$, it is easy to see that for positive integers $m,N$, we have
\[\eta(W_N (mz))=\sqrt{\frac{(N/m)z}{i}}\eta\left(\frac{N}{m}z\right).\]
With this we find that
\begin{align*}
\log \eta_p(\gamma z) %
& = \frac 12\left[\log\eta(W_p\tilde\gamma z)+\log\eta(p\cdot W_p\tilde\gamma z) \right]\\ 
& = \frac 12\left[\log\eta(p\cdot\tilde\gamma z)+\log \eta(\tilde\gamma z)\right]+\frac 14\left[\log\left(p\frac{\tilde\gamma z}{i}\right)+\log\left(\frac{\tilde\gamma z}{i}\right)\right]\\
& = \log\eta_p( z)+\frac 12\sign(-\sqrt pa)^2\log\left(\frac{-\sqrt{p}a z-\sqrt{p}b}{i\sign(-\sqrt{p}a)}\right)+\frac{\pi i}{12}\Phi_p(\tilde\gamma) \\
& \hspace{2in} + \frac 14\left[\log\left(p\frac{\tilde\gamma z}{i}\right)+\log\left(\frac{\tilde\gamma z}{i}\right)\right].
\end{align*}
Therefore the claim follows as soon as we show that 
\begin{multline}\label{signid}
\sign(-\sqrt pa)^2\log\left(\frac{-\sqrt{p}a z-\sqrt{p}b}{i\sign(-\sqrt{p}a)}\right) + \frac 12\left[\log\left(p\frac{\tilde\gamma z}{i}\right)+\log\left(\frac{\tilde\gamma z}{i}\right)\right]\\
= \sign(c)^2\log\left(\frac{c z+d}{i\sign(c)}\right)-\frac{\pi i}{2}\sign(-ac).
\end{multline}
If $a=0$, then we must have $bc=-1$ (so, in particular, $c\neq 0$) and hence, since $\tilde\gamma\in\Gamma_0(p)$, $b=\pm 1/\sqrt{p}$. The left-hand side of \eqref{signid} then simplifies to
\begin{align*}
\log\left(\frac{c z+d}{-ipb}\right)+\frac 12\log p 
& =\log\left(\frac{c z+d}{i\sign(c)}\right)+\log\left(\frac{-\sign(c)}{pb}\right)+\frac 12\log p \\ 
& = \log\left(\frac{c z+d}{i\sign(c)}\right),
\end{align*}
which is the simplified right-hand side of \eqref{signid}. The case $c=0$ is similar. Finally, if $ac\neq 0$, then the left-hand side of \eqref{signid} becomes
\begin{align*}
& \log\left(\frac{-\sqrt{p}a z-\sqrt{p}b}{i\sign(-a)}\right)+\log\left(\frac{c z+d}{i(-\sqrt{p}a z-\sqrt{p}b)}\right)+\frac 12\log p\\
 = & \log\left(\frac{c z+d}{i\sign(c)}\right)+\log\left(\frac{\sign(c)}{i\sign(-a)}\right) %
= \log\left(\frac{c z+d}{i\sign(c)}\right)-\frac{\pi i}{2}\sign(-ac),
\end{align*}
 which completes the proof.
\end{proof}

\begin{remark*}
As shown in \cite[Theorem 12]{JST} (see also \cite[Theorem 1.3]{vPSV}), the function $\eta_p$ defined above (respectively some appropriate power of it) is exactly the function one encounters as the natural analogue of $\eta$ in the Kronecker limit formula for the group $\Gamma_0^+(p)$. 

In fact the cited works prove a Kronecker limit formula for all groups $\Gamma_0^+(N)$ where $N$ is square-free. These groups are generated by the congruence subgroup $\Gamma_0(N)$ together with the so-called \emph{Atkin-Lehner involutions}.  Thus one obtains very similar looking Rademacher symbols associated to these groups as well (see \cite{Goldstein73}). However, we decided to focus on the simplest case where $N=p$ is a prime in this work.
\end{remark*}

\section{Based Edge Paths} \label{TriInHyperbolicPlane}

The modular group $\PSL_2(\Z)$ is a discrete subgroup of $\PSL_2(\R)$; as such, it acts by fractional linear transformations of the upper half plane model of the hyperbolic plane. With this action, much is known about the interactions of $\PSL_2(\Z)$ and the upper half-plane model $\HH$ of the hyperbolic plane. A well known triangulation of the upper half-plane $\HH$ by ideal triangles is made by successive reflections of the ideal triangle with vertices at $0$, $1$, and $\infty$ on the boundary of $\HH$. We denote this triangulation by $K$ (see \cite{HatcherBook} for some important properties of $K$). The vertices of this triangulation $K$ are $\Q \cup \{\infty\}$, where we use the common convention that the fraction notation $\frac 10$ (or occasionally $\frac{-1}{0}$) identifies the point $\infty$. An edge in $K$ joins two fractions $\frac{n_1}{d_1}$ and $\frac{n_2}{d_2}$ (in lowest terms on the boundary of $\HH$) if and only if $n_1 d_2 - n_2 d_1 = \pm 1$.  

Directed edges of $K$ can be identified with the elements of $\PSL_2(\Z)$, and the action of $\PSL_2(\Z)$ on $\HH$ induces a simplicial action on $K$, which corresponds to left multiplication on the edges of $K$. For example, the directed edge from $\infty$ to $0$ will be notated with the order pair $I= \left( \frac{1}{0} , \frac{0}{1} \right)$, and its reverse direction $\left( \frac{0}{1} , \frac{-1}{0} \right)$. We also orientate the ideal triangles of $K$ in a counterclockwise direction. 

Now we define a based (directed) edge path $\alpha$ in $K$ as a path that start with the based edge $I = E_0= \left( \frac{1}{0} , \frac{0}{1} \right)$ and the second edge in $K$ of the path $\alpha$ is $E_1 = \left( \frac{0}{1} , \frac{-1}{d_2} \right)$. We continue, if at the $m^\text{th}$ edge of $\alpha$, say $E_m = \left( \frac{ n_{m-1}}{ d_{m-1} } , \frac{n_m}{d_m} \right)$ in $K$, then the next edge would be $ E_{m+1} = \left( \frac{ n_{m}}{ d_{m} } , \frac{n_{m+1} }{d_{m+1} } \right)$ in $K$. So given a based edge path $\alpha$ in $K$, its edges yield a sequence of endpoints starting at $ \frac{1}{0} , \frac{0}{1} , \ldots, \frac{ n_{k-1}}{ d_{k-1} }$, and ending at $\frac{ n_{k}}{ d_{k} } $.  Figure \ref{ABasedPath} shows an example of a based edge path $\alpha$, which gives the sequence of endpoints $ \frac{1}{0} , \frac{0}{1} , \frac{ 1 }{ 2 }, \frac{1}{3}$, and ending at $\frac{ 3 }{ 8 } $.

Furthermore, we can make a based edge path $\alpha$ from the list of integers $\left( a_1 , \ldots , a_k \right)$ in a geometric way from the triangulation $K$. Figure \ref{ABasedPath} shows an example of a based edge path $\alpha$ from the list of integers $(-2, 1, -2)$.

\begin{figure}[h] 
\setlength{\unitlength}{80mm}
\begin{center}
\begin{picture}(1,0.70945946)%
    \put(0,0){\includegraphics[width=\unitlength]{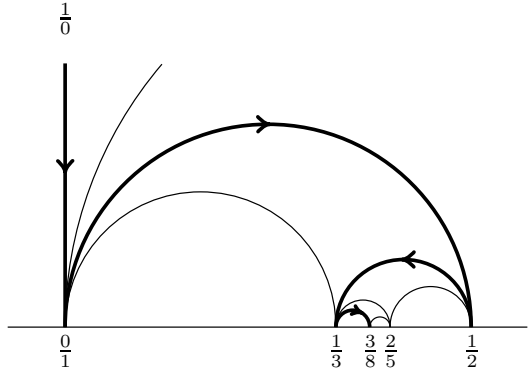}}%
    \put(0.16328589,0.57290809){\labN{$\frac{1}{0}$}}%
    \put(0.16322971,0.1352088){\labnhS{$\frac{0}{1}$}}%
    \put(0.61349183,0.13517308){\labnhS{$\frac{1}{3}$}}%
    \put(0.67011859,0.13520803){\labnhS{$\frac{3}{8}$}}%
    \put(0.70380505,0.13520885){\labnhS{$\frac{2}{5}$}}%
    \put(0.83883528,0.13511663){\labnhS{$\frac{1}{2}$}}%
\end{picture}%
\end{center}
\caption{Example of the based edge path $\alpha = (-2, 1, -2)$}
\label{ABasedPath}
\end{figure}

A geometric way that a based edge path $\alpha$ in $K$ is built from a list of integers $\left( a_1 , \ldots , a_k \right)$ is as follows: First, from $\alpha$'s starting based edge $I=E_0 = \left( \frac{1}{0}, \frac{0}{1} \right)$, if $ a_1 > 0 $, turning to the left (if $a_1 < 0$, turning to the right) through the common vertex $\frac{0}{1}$ going $\abs{a_1}$ number of triangles of $K$ to get to the next edge $E_1 = \left( \frac{0}{1} , \frac{-1}{d_2} \right)$ of the based path $\alpha$. Now continuing (for $j \geq 2$), from $E_{j-1}= \left( \frac{ n_{j-1} }{ d_{j-1} } , \frac{ n_{j} }{ d_{j} } \right)$ turning to the left, if $a_{j} > 0 $, (to the right, if $a_{j} < 0 $) through the common vertex $\frac{ n_{j} }{ d_{j} }$ going $\abs{a_{j}}$ number of triangles of $K$ to get to the next edge $E_j = \left( \frac{ n_{j} }{ d_{j} } , \frac{ n_{j+1} }{ d_{j+1}  } \right)$, where the signs of $n_{j+1}$ and $d_{j+1}$ are chosen so that $n_j d_{j+1} - n_{j+1} d_j = 1$.

In fact, we can obtain the integer for the direction and number of triangles of $K$ to turn through from the three endpoints from one edge to the next edge in the path.  The proof of Lemma 1.9 in \cite{KirbyMelvin94} shows that if the three consecutive endpoints of the edge $E_{j-1}$ to the edge $E_j$ are $\frac{ n_{j-1} }{ d_{j-1} } , \frac{ n_{j} }{ d_{j} },  \frac{ n_{j+1} }{ d_{j+1} } $ then the integer $a_j$ for the direction and number of triangles of $K$ to turn through can be found via the identities
\[
\abs{a_j} = \abs{ n_{j-1} d_{j+1} - n_{j+1} d_{j-1} },\quad\text{and}\quad \sign(a_j)=\sign\left(n_jd_{j+1}-n_{j+1}d_j\right).
\]

This relationship between the endpoints of consecutive edges and the geometry of turns in a based path can establish the next lemma (from \cite{KirbyMelvin94}, see Lemma 1.9), which relates the final edge (in reverse direction to a matrix in $\PSL_2(\Z)$) of a based edge path $\alpha$ to a product in terms of the standard generators of the modular group via the based edge path $\alpha = \left( a_1 , \ldots , a_k \right)$. 

\begin{lemma*}[\cite{KirbyMelvin94}] \label{PathAsMatrix}
Let $\alpha = ( a_1 , \ldots , a_k)$ be a based edge path in $K$ whose final (directed) edge goes from $\frac{b}{d}$ to $\frac{a}{c}$ where $ ad - cb = 1$. As elements of $\PSL_2(\Z)$, $
\abcd = S(T^{a_1} S) \cdots (T^{a_k} S)$, where $S = \left( \begin{smallmatrix} 0 &  -1 \\ 1 & 0 \end{smallmatrix} \right)$ and $T =  \left( \begin{smallmatrix} 1 &  1 \\ 0 & 1 \end{smallmatrix} \right)$.
\end{lemma*}

\section{Two based paths for elements of the congruence subgroup $\Gamma_0(p)$} \label{ProofGeomThm}

Like in section \ref{Symbolp}, we set $p$ to be an odd prime. We start by considering elements of the congruence subgroup $\Gamma_0(p)$. When viewing an element from $\Gamma_0(p)$ as an edge in $K$, the endpoints can be multiplied by $p$ and produce another edge of $K$. It is important to keep in mind that multiplication by $p$ as a M\"{o}bius transformation does not map $K$ to itself. However, it does map the base edge $\left( \frac{1}{0}, \frac{0}{1} \right)$ to itself and maps edges from elements of $\Gamma_0(p)$ to edges in $K$. With $\gamma = \abcd \in \Gamma_0(p)$, (that is, $ c\equiv 0\pmod p $), we see that conjugating $\gamma$ by the M\"{o}bius transformation multiplication by $p$ is also a edge in $K$; that is,  
\[
\begin{pmatrix} %
p/\sqrt{p} & 0 \\ %
0 &  1/\sqrt{p}  %
\end{pmatrix} %
\begin{pmatrix} %
a   & b \\ %
c   & d  %
\end{pmatrix} %
\begin{pmatrix} %
1 /\sqrt{p} & 0 \\ %
0           & p/\sqrt{p}  %
\end{pmatrix} %
= %
\begin{pmatrix} %
a   & pb \\ %
c/p & d %
\end{pmatrix}.
\]
We make two based paths $\alpha= (a_1, \ldots , a_k)$ and $\beta = (b_1 , \ldots , b_m) $ so that 
\[
\bigabcd =  S(T^{a_1} S) \cdots (T^{a_k} S) %
\quad \text{and} \quad %
\begin{pmatrix} %
a   & pb \\ %
c/p & d %
\end{pmatrix} = S(T^{b_1} S) \cdots (T^{b_m} S).
\]
These two edges in $K$ are the two ending edges (in reverse direction) for the based paths $\alpha$ and $\beta$. We refer to the based edge paths $\alpha$ and $\beta$ as being made from $\gamma = \abcd \in \Gamma_0(p)$.  

%

\begin{proof}[Proof of Theorem~\ref{thmmain}]
Let $\gamma = \abcd \in \Gamma_0(p)$, (that is, $ c\equiv 0\pmod p $); we make two based edge paths $\alpha= (a_1, \ldots , a_k)$ and $\beta = (b_1 , \ldots , b_m) $ where the ending edges (in reverse direction) are $\abcd$ and $\left( \begin{smallmatrix} %
a   & pb \\ %
c/p & d %
\end{smallmatrix} \right)$, respectively.  

Now we apply the Theorem of Kirby and Melvin (in the introduction) to compute $\Phi$ of these two matrices using the geometric formulation of Rademacher symbol 
\[
\Phi(\gamma) = \tau_\alpha - 3 \sigma_\alpha \quad \text{and} \quad  \Phi \left( \begin{pmatrix} %
a   & pb \\ %
c/p & d %
\end{pmatrix} \right) = \tau_\beta - 3 \sigma_\beta,
\]
where $\tau_\alpha$, $\sigma_\alpha$, $\tau_\beta$, $\sigma_\beta$ are as defined in the statement of Theorem~\ref{thmmain}.  Now the result follow by applying Theorem~\ref{thmPhiP}.  
\end{proof}



\end{document}